\documentclass[10pt, article]{amsart}
\usepackage[left]{showlabels}

\usepackage{tikz}
\usetikzlibrary{calc}

\usepackage{ae} 
\usepackage[T1]{fontenc}
\usepackage[cp1250]{inputenc}
\usepackage{amsmath}
\usepackage{amssymb, amsfonts,amscd,verbatim}
\usepackage{hyperref}

\usepackage{indentfirst}
\usepackage{latexsym}
\input xy
\xyoption{all}

\usepackage{amsmath}    

\theoremstyle{plain}
\newtheorem{Prop}{Proposition}[section]
\newtheorem{Thm}[Prop]{Theorem}

\newtheorem{Lem}[Prop]{Lemma}

\theoremstyle{definition}
\newtheorem{Def}[Prop]{Definition}

\theoremstyle{remark}
\newtheorem{Rem}[Prop]{Remark}

\newtheorem{Example}[Prop]{\bf Example}

\newcommand{\im}{\operatorname{Im}\nolimits}

\newcommand{\diam}{\operatorname{Diam}\nolimits}

\newcommand{\Rips}{\operatorname{Rips}\nolimits}
\newcommand{\cRips}{\operatorname{\overline{Rips}}\nolimits}
\newcommand{\Cech}{\operatorname{Cech}\nolimits}
\newcommand{\cCech}{\operatorname{\overline{Cech}}\nolimits}

\newcommand{\supp}{\operatorname{supp}\nolimits}
\newcommand{\rank}{\operatorname{rank}\nolimits}

\def\RR{{\mathbb R}}

\def\NN{{\mathbb N}}

\def\f{{\varphi}}
\def\U{\mathcal{U}}
\def\C{\mathcal{C}}
\def\R{\mathcal{R}}
\def\V{\mathcal{V}}
\def\W{\mathcal{W}}
\def\N{\mathcal{N}}
\def\eps{\varepsilon}

\numberwithin{equation}{section}
\title[Rips complexes as nerves and...]
{Rips complexes as nerves and a Functorial Dowker-Nerve Diagram}

\author{\v Ziga ~Virk}
\address{University of Ljubljana, Slovenia}
\email{ziga.virk@fri.uni-lj.si}
\thanks{}

\begin{document}

\maketitle
\begin{center}
\today
\end{center}

\begin{abstract}
Using ideas of the Dowker duality we prove that the Rips complex at scale $r$  is homotopy equivalent to the nerve of a cover consisting of sets of prescribed diameter. We then develop a functorial version of the Nerve theorem coupled with the Dowker duality, which is presented as a Functorial Dowker-Nerve Diagram. These results are incorporated into a systematic theory of filtrations arising from covers. As a result we provide a general framework for reconstruction of spaces by Rips complexes,  a short proof of the reconstruction result of Hausmann, and completely classify reconstruction scales for metric graphs.  Furthermore we introduce a new extraction method for homology of a space based on nested Rips complexes at a single scale, which requires no conditions on neighboring scales nor the Euclidean structure of the ambient space.
\end{abstract}

\section{Introduction}

One of the most common approaches in mathematics is approximation: in order to study an object we can often represent or approximate it by simpler objects, of which we have better understanding. In the context of topology, the approximations of spaces are most frequently performed using \v Cech or Rips simplicial complexes. Each such complex represents a snapshot of a space at a chosen scale.  Taking into account all scales, the individual complexes can be bound together to form a filtration of a space. These filtrations have been initially used to study local properties of spaces: they are used to define the \v Cech (co)homology and form a basis of shape theory \cite{MS} (see \cite{Dow} for more on this perspective). The Rips complexes, originally discovered by Vietoris, were rediscovered by Rips in the context of geometry of  groups. In the asymptotic setting both mentioned filtrations have been used to provide approximations of coarse spaces. 
Our focus however will be on the perspective of computational topology: filtrations are used to reconstruct spaces at small scales and to compute persistent homology at all scales, which reflects the size of holes in the space.

The choice of a filtration, based on  \v Cech or Rips complexes, obviously matters. While both filtrations are interleaved (meaning that the limiting objects as scale goes to $0$ of $\infty$ do not depend on the choice), the complexes at each scale differ, and precise connection between is known only in a few settings \cite{AA, ZV}. \v Cech complexes are nerves of covers hence their homotopy type can be usually interpreted in a convenient way as neighborhoods of a space using the Nerve theorem. Consequently, the corresponding theory is rich  and contains strong reconstruction results \cite{W}. Rips complexes on the other hand are easier to define and compute, making them the favorite choice for applications. However, there is no clear interpretation of their homotopy type, and the reconstruction results \cite{Haus, Lat, Att} require significantly more structure and hold only for small scales. 

In this paper we provide a systematic treatment of nerve and Vietoris complexes arising from covers, of which the \v Cech and Rips complexes are a special case.  
The main results are the following:
\begin{itemize}
 \item Rips complexes as nerves (Theorem \ref{Thm:RipsNerve}): we can express the homotopy type of  Rips complex  as a nerve of a cover by sets of prescribed diameter. 
 \item Functorial Dowker-nerve diagram (Theorem \ref{Thm:NerveNested}): we present a functorial version of the Nerve theorem and Dowker duality, jointly incorporated into a diagram.
\end{itemize}

The ability to treat the homotopy type of Rips complexes as nerves (in a functorial way) allows us to gain an insight into reconstruction results:
\begin{itemize}
 \item we provide a short proof of Hausmann's reconstruction for Riemannian manifolds \cite{Haus}, and extend it to persistent setting and closed Rips complexes;
 \item we completely classify reconstruction scales for metric graphs;
 \item we present a general way of extracting homology of a space using nested Rips complexes at a single scale arising from a single good cover of a space. 
\end{itemize}

The structure of the paper is the following. Section \ref{Sec:PartUnity} provides preliminaries on topologies on infinite complexes and partitions of unity.
Section \ref{Sect:Dowker} contains a systematic study of complexes arising from covers, Dowker duality,  and contains Nerve theorem for Rips complexes. In
Section \ref{Sect:Nerve} we use results of the previous section to develop reconstruction results using Rips complexes, in particular for  Riemannian manifolds and metric graphs.
Section \ref{Sect:Funct} contains functorial Dowker-Nerve diagram and applications to reconstruction.
In section Section \ref{Sect:subsets} we describe how to extract homology of a space from nested Rips complexes of finite subsets and provide a corresponding example.
Section \ref{sect:App} (Appendix) demonstrates how to use the Nerve theorem to reconstruct the homotopy type of Rips complexes  and the  restrictions connected to it.

\section{Preliminaries on complexes and partitions of unity}
\label{Sec:PartUnity}

In this section we present background on infinite simplicial complexes and partitions of unity. Our interest in the later stems from the fact that appropriate partitions of unity represent convenient maps to simplicial complexes. We provide several technical results that will be used in the future sections. The two interpretations of a simplicial complex (abstract and geometric simplicial complex) will be used interchangeably when no confusion may occur. 

\subsection{Infinite simplicial complexes}

Given $n \geq 0$, the standard (geometric) $n$-simplex is the convex hull of the collection of $n+1$  points of the form $(0,0,\ldots, 0,1,0,\ldots, 0)\in \RR^{n+1}$.
Suppose $K$ is a (geometric) simplicial complex. 
We usually equip $K$ with a topology, for which the restriction onto each simplex coincides with the topology on the standard simplex. If $K$ is locally finite, meaning that each of its vertices is contained in only finitely many simplices, then there is a unique such topology on $K$. However, if $K$ is not locally finite, there are more topologies to choose from. We recall two of them that appear standardly in textbooks:
\begin{enumerate}
 \item \textbf{The weak topology}. It is obtained by gluing simplices together via quotient maps along their boundaries, hence the resulting topology is the weak topology with respect to the identification of the boundaries. This is the standard topology used for infinite simplicial or CW (also called cell) complexes. Unless stated otherwise, each simplicial complex will be equipped with this topology. A convenient aspect of this topology is related to constructions of maps defined on it \cite[Theorem 2, p. 290]{MS}: if $K$ is a simplicial complex and $f\colon K \to X$ is a map, then $f$ is continuous iff the restriction $f|_\sigma$ is continuous for all simplices $\sigma$ of $K$.
 \item \textbf{The metric (or $\ell_1$ or strong) topology}, see for example \cite[Remark 7.14, p.116]{Dol}, \cite[Section 5]{Dyd}, \cite{MS} or \cite{Sak}. In order to define it we first have to introduce the barycentric coordinates. Let $K^{(0)}=\{v_i\}_{i\in J}$ be the set of vertices of $K$. Note that the points in the standard simplex are presented as  convex combinations of the vertices of that simplex. Similarly, each simplex $\sigma=[v_0, v_1, \ldots, v_n]$ in a simplicial complex $K$ can be parameterized by the barycentric coordinates $(t_0, t_1, \ldots, t_n) \mapsto \sum_i t_i v_i$ with $\sum_i t_i=1$ and $t_i\in [0,1], \forall i$. For example, vertex $v_0$ has only one non-zero barycentric coordinate: $t_0=1$. The midpoint on the edge between $v_0$ and $v_1$ has two non-zero barycentric coordinates: $t_0=t_1=1/2$. For each $i\in J$ let $\lambda_i \colon K \to [0,1]$ denote the $i^{th}$ barycentric coordinate, i.e., $\lambda_i (\sum_j t_j v_j)=t_i$. The strong topology on $K$ is the coarsest topology, under which all barycentric coordinates are continuous. More conveniently, we may define a metric on $K$ using the barycentric coordinates: 
 $$
 d_{\ell_1}\Big(\sum_i t_i v_i, \sum_i t'_i v_i \Big) = \sum_i |t_i - t'_i |.
 $$
 Metric $d_{\ell_1}$ is actually obtained by embedding $K$ into the metric space $\ell_1(K_0)$ so that the restriction to each simplex is a linear map.
 The resulting metric space will be denoted by $K_m$. 
 A convenient aspect of this topology is related to constructions of maps to it \cite[Theorem 8, p. 301]{MS}: a map $f\colon X \to K_m$ is continuous iff  $\lambda_i \circ f$ is continuous for all $i\in J$.
  \end{enumerate}

Conveniently enough, the identity map $K \to K_m$ is a homotopy equivalence (\cite[Corollary 2.9, p. 354]{Dol} or \cite[Theorem 4.9.6]{Sak}). 

Maps $f,g \colon X \to K$ are \textbf{contiguous}, if for each $x\in X$, there exists a simplex in $K$ containing both  $f(x)$ and $g(x)$. Contiguous maps are homotopic (\cite[Remark 2.22, p. 359]{Dol} or \cite[Theorem 4.9.7]{Sak}), although the obvious straight-line homotopy is only continuous in $K_m$.

\subsection{Partitions of unity}

A \textbf{partition of unity} on a topological space $X$ is a collection of continuous functions $\{f_i \colon X \to [0,1]\}_{i\in \mathcal{I}}$ such that for each $x\in X$ we have $\sum_{i\in \mathcal{I}} f_i(x)=1$. Such a partition is locally finite, if for each $x\in X$ there exist a neighborhood $U_x$ of $x$ and a finite $\mathcal{I}_x \subset \mathcal{I}$, such that $f_i(y)=0, \forall y\in U_x, \forall i\in \mathcal{I}\setminus \mathcal{I}_x$. A partition of unity is point-wise finite, if for each $x\in X$ only finitely many $f_i(x)$ are non-zero.

Interpretations of partitions of unity as barycentric coordinates and vice versa provide a correspondence between appropriate partitions of unity and maps to simplicial complexes. Point-wise finite partitions of unity are in bijective correspondence with continuous maps to metric simplicial complexes. 
A locally finite partition of unity on a space represents a continuous map to a simplicial complex in the weak topology. See \cite[Theorem 6.5]{Dyd} for more details of such correspondences.

We will be interested in partitions of unity arising from a cover. For a  continuous function $f\colon X \to [0,1]$ define support $supp(f)$ as the closure of the subset $f^{-1}(0,1]$ in $X$.  Given a cover $\U=\{U_j\}_{j\in A}$ of $X$, a  partition of unity $\{f_i \colon X \to [0,1]\}_{i\in \mathcal{I}}$ on $X$ is \textbf{subordinated} to $\U$, if for each ${i\in \mathcal{I}}$ there exists $j_i\in A$ so that $\supp(f_i)\subset U_{j_i}$. We will be particularly interested in locally-finite partitions of unity on $X$ subordinated to a cover $\U$ on $X$, as these encode maps to the nerve of $\U$, i.e., maps $X \to \N(\U)$. (For a definition of the nerve see Definition \ref{Def:ComplexesOfCovers} provided in Section \ref{Sect:Dowker} within a broader context of constructions of complexes.) 
\begin{itemize}
\item  A cover $\U$ of $X$ is \textbf{numerable}, if it admits a locally finite partition of unity on $X$ subordinated to $\U$.
\item Space $X$ is \textbf{paracompact}, if each open cover of $X$ is numerable. Consequently, a closed cover of a paracompact space $X$ is numerable if the interiors of the elements of $\U$ cover $X$. The class of paracompact spaces includes all metric spaces.  
\end{itemize}

Suppose $\U$ is a numerable cover of $X$. We can define a map $X \to \N(U)$ by choosing a partition of unity subordinated to $\U$ and declaring it represents barycentric coordinates. It is easy to verify (see for example \cite[Corollary 4.9.2, p. 198]{Sak}) that a different choice of a partition of unity induces a different but contiguous (hence homotopic) map. Consequently we will denote by $i_\U \colon X \to \N(\U)$ any such map and keep in mind that its homotopy type does not depend on the choice of the generating partition of unity.

\section{Dowker duality and complexes arising from covers}
\label{Sect:Dowker}

In this section we provide a systematic introduction of complexes arising from covers, and conclude with a statement describing the homotopy type of a Rips complex as a nerve of a specific cover. Our main tool will be the Dowker duality (Theorem \ref{Thm:Dowker}) introduced in \cite{Dow}.  

Suppose $R\subset Y \times Z$ is a subset of a product. It can be thought of as a relation \cite{Dow}. For our setting though it will be convenient to think of the corresponding (usually infinite) binary matrix $M_{R}=(m_{y,z})_{y\in Y, z\in Z}$ defined by the following rule:
\begin{itemize}
 \item $m_{y,z}=1$ if $(y,z)\in R$;
 \item $m_{y,z}=0$ else.
\end{itemize}
For any $z_0\in Z$ let $C_{z_0}=(m_{y,{z_0}})_{y\in Y}$ denote the column corresponding to $z_0$. Similarly, for any $y_0\in Y$ let $R_{y_0}=(m_{{y_0},z})_{z\in Z}$ denote the row corresponding to $y_0$.

Following \cite{Dow} we introduce two complexes associated to a binary matrix $M_R$.

\begin{Def}\label{Def:ComplexesOfMatrix}
Suppose $R\subset Y \times Z$, and $M_R$ is the corresponding binary matrix. 
 
 The \textbf{column complex}  of $M_R$ is a simplicial complex $\C(M_R)$ defined by the following conditions:
\begin{enumerate}
 \item the set of vertices of $\C(M_R)$ consists of all elements of $z\in Z$, for which the column $C_z$ has a non-zero entry;
 \item a finite subset $\sigma \subset Z$ is a simplex of $\C(M_R)$ iff there exists $y\in Y$, for which $m_{y,z}=1, \forall z\in \sigma$.
\end{enumerate}

Alternatively we could define $\C(M_R)$ by the following rule: a finite subset $\sigma \subset Z$ is a simplex in $\C(M_R)$ iff there exists $y\in Y$ so that $\sigma$ appears is a finite subset of $R \cap ({\{y\} \times Z})$.

In an analogous manner we can define  the \textbf{row complex} $\R(M_R)$, although a quicker definition would be using a transposed matrix: $\R(M_R)=\C(M_R^T)$.
\end{Def}

\begin{Thm}\label{Thm:Dowker}
 [Dowker duality  \cite{Dow}] Suppose $R\subset Y \times Z$ and $M_R$ is the corresponding binary matrix. Then $\C(M_R)$ and $\R(M_R)$ are homotopy equivalent.
\end{Thm}

Dowker duality can be proved easily using the Nerve Lemma, stated in this paper as Theorem \ref{Thm:Nerve}, by considering the following steps:
\begin{itemize}
	\item  Columns of $M_{R}$ represent a good cover of $\R(M_{R})$ in the sense that for each $z\in Z$, the set $U_z = \{y\in Y: (y,z)\in R\}$ spans a full contractible complex $K_{z}$ (i.e., each finite subset of $U_z$ is a simplex in $K_z$) in $\R(M_{R})$. While the collection of complexes $\{K_z\}_{z\in Z}$ is a closed cover of $\R(M_{R})$, we can thicken each of them slightly to an open neighborhood by preserving the intersection properties (i.e. the nerve) of $\{U_z\}_{z\in Z}$.
	\item By definition, $\C(M_R)$ is the nerve of $\{U_z\}_{z\in Z}$.
\end{itemize}

Our aim is to apply Theorem \ref{Thm:Dowker} to the case of complexes arising from covers. 
Following the terminology of \cite{Dow} we introduce two complexes associated to a collection of subsets of $X$. Before we do that we would like to bring attention to a technical detail. 

\begin{Rem}
 We will often be talking about collection of subsets, or covers, the later being a collection of subsets, whose union is the whole space. Such collections may contain multiple copies of the same set, even of the empty set. For example, a collection $\{\{1,2\},\emptyset, \{2\},\{2\}\}$ is a collection of subsets of $\{1,2\}$. We will use the same notation for collections and sets, and the context should make sure there is no confusion. For example, if $\U=\{U_\alpha\}_{\alpha\in A}$ is a cover of $X$, then $U_\alpha$ are subsets of $X$, but there may exist $\alpha_1, \alpha_2, \alpha_3\in A$ with $U_{\alpha_1}=U_{\alpha_2}$ or $U_{\alpha_3}=\emptyset$.
 \end{Rem}

\begin{Def}\label{Def:ComplexesOfCovers}
 Suppose $\U$ is a collection subsets of $X$. 
 
 The \textbf{nerve}  of $\U$ is the simplicial complex $\N(\U)$ defined by the following conditions:
\begin{enumerate}
 \item the set of vertices of $\N(\U)$ consists of all non-trivial elements of $\U$;
 \item a finite subset $\sigma \subset \U$ is a simplex of $\N(\U)$ iff $\cap_{U\in \sigma} U \neq \emptyset$.
\end{enumerate}
Strictly speaking, (2) implies (1).

The \textbf{Vietoris complex}  of $\U$ is the simplicial complex $\V(\U)$  defined by the following conditions:
\begin{enumerate}
 \item the set of vertices of $\V(\U)$ consists of all points of $X$, which are contained in some element of $\U$;
 \item a finite subset $\sigma \subset X$ is a simplex of $\V(\U)$ iff there exists $U\in \U$ so that $\sigma \subset U$.
\end{enumerate} 
\end{Def}

\begin{Rem}
 It is easy to see that each simplicial complex $K$ can be expressed as a nerve (using open stars of vertices as a cover of  $K$) and as a Vietoris complex (using the cover consisting of all $n$-tuples of vertices that span a simplex in $K$ as a cover of the collection of all vertices of $K$).
\end{Rem}

In order to connect the nerve and the Vietoris construction to the Dowker duality we  assign a binary matrix to a cover.

\begin{Def}
 \label{Def:matrix}
 Suppose $\U=\{U_{i}\}_{i\in A}$ is a cover of $X$. The associated binary matrix  $M_{\U}=(a_{x,U})_{x\in X, U\in \U}$
 is defined by 
 \begin{itemize}
 \item $a_{x,U}=1$ if $x\in U$;
  \item $a_{x,U}=0$ else.
\end{itemize}
\end{Def}

Matrix $M_{\U}$ can also be thought of as a subset of $X \times A$ encoding a binary relation of containment between the points of $X$ and the elements of the cover $\U$.

Note that each column $C_{U}=\{a_{x,U}\}_{x\in X}$ of $M_{\U}$ encodes an element  $U\in \U$: the non-zero entries correspond precisely to the elements of $U$. 
Similarly, each row $R_{x}=(a_{x,U})_{U\in \U}$ of $M_{\U}$ locates an element  $x\in X$: the non-zero entries correspond precisely to the sets from $\U$ containing $x$. 

\begin{Prop}
  Suppose $\U=\{U_{i}\}_{i\in A}$ is a cover of $X$. Then $\N(\U)=\C(M_{\U})$ and $\V(\U)=\R(M_\U)$. In particular, $\N(\U)\simeq \V(\U)$.
\end{Prop}

In metric spaces, two specific constructions of complexes feature prominently throughout mathematics: the \v Cech complex and the Rips complex.

\begin{Def}
 \label{Def:CechRips}
 Suppose $X$ is a metric space, $A\subset X$ and $r>0$. For $x\in X$ let $B(x,r)$ and $\overline B(x,r)$ denote the open and closed balls respectively.
 
 The (open) \textbf{\v Cech} complex $\Cech_X(A,r)$ is defined by the following rules:
\begin{itemize}
 \item its vertices are elements of $A$;
 \item a finite subset $\sigma \subset A$ is a simplex in $\Cech_X(A,r)$ iff $\cap_{a\in \sigma} B(a,r) \cap X \neq \emptyset$.
\end{itemize}
 The closed \textbf{\v Cech} complex $\cCech_X(A,r)$ is defined analogously using closed balls. 
 
 The (open) \textbf{Rips} complex $\Rips(A,r)$ is defined by the following rules:
\begin{itemize}
 \item its vertices are elements of $A$;
 \item a finite subset $\sigma \subset A$ is a simplex in $\Rips_X(A,r)$ iff $\diam (\sigma)<r$.
\end{itemize}
 The closed \textbf{Rips} complex $\cRips(A,r)$ is defined analogously using the non-strict inequality.
\end{Def}

The \v Cech complex is usually defined as the nerve of the collection of open balls, i.e., for $A\subset X$ metric, $r>0$, and $\U=\{B(a,r)\}_{a\in A}$, $\Cech_X(A,r)=\N(\U)$. Keeping in mind the identification of vertices $a\in A$  (in the \v Cech complex) by $B(a,r)$ (in the column and the nerve complex) for the rest of this paragraph, we observe that the two definitions of the \v Cech complex agree, and that $\Cech_X(A,r)=\C(M_\U)$. In case when $A=X$ we can use this fact to deduce two interesting statements from the Dowker duality:
\begin{enumerate}
 \item $\Cech_X(X,r)=\N(\U)=\C(M_\U)\simeq \R(M_U)= \V(\U)$. Actually, basic geometry implies $\N(\U)= \V(\U)$ (via the above-mentioned identification of vertices) since a collection of balls of radius $r$ intersects iff its centers themself are contained in a ball of radius $r$ (based at any point of the aforementioned intersection). The coincidence $\N(\U)= \V(\U)$ via a natural identification of vertices of both simplices seems to be specific to the cover by open balls.
 \item Suppose $\U'$ is a cover of $X$ with the following property for	each finite subset $\sigma \subset X$:  $\sigma$ is contained in some element of $\U'$ iff $\cap_{a\in \sigma} B(a,r) \cap X \neq \emptyset$. Then $\R(M_\U)=\R(M_{\U'})$ by definition and thus by the Dowker duality, $\Cech_X(X,r)=\N(\U)\simeq \N(\U')$. This trick allows us to change a cover and  keep the homotopy type of the nerve. It will come handy in the context of Rips complexes. 
\end{enumerate}
Analogous conclusions hold for the closed \v Cech complexes as well. 

The Rips complex has, to the best of the author's knowledge, never been expressed as a nerve of a covering. However, using the setup of this section we can do just that. 

\begin{Thm} [Rips complex as a nerve]
\label{Thm:RipsNerve}
Suppose $X$ is a metric space and $r>0$. Let $\U$ be a cover of $X$ with the following   property for	each finite subset $\sigma \subset X$:  $\sigma$ is contained in some element of $\U$ iff $\diam(\sigma)<r$. Then $\Rips(X,r)\simeq \N(\U)$. 

Let $\U'$ be a cover of $X$ with the following   property for	each finite subset $\sigma \subset X$:  $\sigma$ is contained in some element of $\U$ iff $\diam(\sigma)\leq r$. Then $\cRips(X,r)\simeq \N(\U')$. 
\end{Thm}

\begin{proof}
Equality $\Rips(X,r)=\R(M_\U)$ follows by definition, and by Dowker duality we have $\R(M_\U)\simeq \C(M_\U)=\N(\U)$. The same argument works for closed Rips complexes. 
\end{proof}

\begin{Rem} 
\label{Rem:RipsNerveExamples}
 Here we provide some examples of  coverings $\U$ satisfying the conditions of Theorem \ref{Thm:RipsNerve}. 
\begin{enumerate}
	\item Suppose $\U$ is a cover by all finite subsets of $X$ of diameter less than $r$. Then $\Rips(X,r)\simeq \N(\U)$.
	\item Suppose $\U$ is a cover by all  subsets of $X$ of diameter less than $r$. Then $\Rips(X,r)\simeq \N(\U)$.
	\item Suppose $\U$ is a cover by all finite subsets of $X$ of diameter at most $r$. Then $\cRips(X,r)\simeq \N(\U)$.
	\item Suppose $\U$ is a cover by all  subsets of $X$ of diameter at most $r$. Then $\cRips(X,r)\simeq \N(\U)$.
	\item Suppose $X$ is a locally flat compact manifold of dimension $N$. Then there exists $q>0$ so that for each $x\in X$ the ball $B(x,q)$ is isometric to the Euclidean $n$-ball of radius $q$. Assume now that $r< q/4$. For each finite subset $\sigma \subset X$ of diameter at most $r$ we can well define the convex hull of $\sigma$ via the mentioned local isomorphism with the Euclidean ball (see Definition \ref{Def:GeodConvex} for details). Now suppose $\U$ consists of the convex hulls of all finite subsets $\sigma \subset X$ of diameter less than $r$ (it essentially consists of edges, triangles, polygons, tetrahedra, etc...). Then $\Rips(X,r)\simeq \N(\U)$.
\end{enumerate}
More examples will be provided in the next section in the context of the reconstruction results.
\end{Rem}

\section{Reconstruction results}
\label{Sect:Nerve}

In this section we combine the setting of  Section \ref{Sect:Dowker} with the Nerve lemma and provide a number of reconstruction results. To be more precise, we are interested in reconstructing the homotopy type of a space via Rips or \v Cech complexes. An idea of reconstructing spaces via  \v Cech complexes is rather old and well understood. The reconstruction is usually based on the Nerve lemma, the first versions of which were apparently proved by Borsuk \cite{Bor48} and Leray \cite{Ler45}. More results aimed at the computational setting were proved in this millenium \cite{W, CCL09}. An absence of formulation of the Rips complex in terms of a nerve made similar reconstruction results for Rips complexes much more rare and not as general as the results for the \v Cech complexes \cite{Haus, Lat, Att}. With Theorem \ref{Thm:RipsNerve}  we aim to bridge this gap of understanding.

\begin{Def}
 A cover $\U$ of a space $X$ is \textbf{good}, if the intersection of each finite subcollection of $\U$ is contractible or empty. 
\end{Def}

The version of the Nerve Lemma suitable for this paper will be derived from the following result.

\begin{Thm}
 \label{Thm:Dieck}
 [Theorem 1 of \cite{Die} with $B$ being a singleton] 
 Let $f\colon X \to Y$ be a map. Suppose $\U=\{U_\alpha\}_{\alpha \in A}$ and $\V=\{V_\alpha\}_{\alpha \in A}$ are numerable covers of $X$ and $Y$ respectively, so that $f(U_\alpha)\subset V_\alpha$ for each $\alpha\in A$. Assume that for each finite $\sigma \subset A$ the restriction $f|_{\cap_{\alpha \in \sigma}U_\alpha}\colon \cap_{\alpha \in \sigma} U_\alpha \to \cap_{\alpha \in \sigma}V_\alpha$ is a homotopy equivalence. Then $f$ is a homotopy equivalence.
\end{Thm}

For the following theorem recall the definition of map $i_\U \colon X \to \N(\U)$ and properties of numerable covers from the last part of Section \ref{Sec:PartUnity}.

\begin{Thm}
[Nerve Lemma]
\label{Thm:Nerve}
 If $\U$ is a good numerable cover of a  space $X$, then $i_\U \colon X\to\N (\U)$ is a homotopy equivalence.
\end{Thm}

\begin{proof}
 Apply Theorem \ref{Thm:Dieck} with $f$ being $i_\U$ and $\V$ being the cover of $\N(\U)$ by the collection of open stars of  vertices.
\end{proof}

A standard application of the Nerve lemma is the following reconstruction result via the \v Cech complex: if for some $r>0$ a cover by open $r$-balls of a paracompact space $X$ is a good cover, then the corresponding \v Cech complex is homotopy equivalent to $X$. Using Theorem \ref{Thm:RipsNerve} we can state an equivalent result for the Rips complexes. 

\begin{Thm} \label{Thm:RipsNerve1}
[Reconstruction Theorem for Rips complexes]
 Suppose $r>0$ and $\U$ is a good numerable cover of a  space $X$ with the following property for each finite subset $\sigma \subset X$:  $\sigma$ is contained in some element of $\U$ iff $\diam(\sigma)<r$ (or $\diam(\sigma) \leq r$ respectively). Then $\Rips(X,r)\simeq X$ (or $\cRips(X,r)\simeq X$ respectively). 
 
 In particular, if $\U$ is a good $\Rips_r$-cover of $X$, then $\Rips(X,r)\simeq X$. 
\end{Thm}

\begin{proof}
 Follows from Theorems \ref{Thm:RipsNerve} and \ref{Thm:Nerve}.
\end{proof}

\begin{Example}
 If $X$ is a compact locally flat manifold, then cover $\U$ of Remark \ref{Rem:RipsNerveExamples} (4) is a good cover and hence $X \simeq \Rips(X,r)$.
\end{Example}

\subsection{Reconstruction of Riemannian manifolds and more general geodesic spaces}

A reconstruction theorem for Riemannian manifolds via Rips complexes was first proved in \cite{Haus}. Using our framework we can provide a much simpler proof. Recall that a metric space is geodesic if every pair of points $x,y$ is connected by a path of length $d(x,y)$ called a geodesic. Note that our scope of a geodesic is a bit more restrictive than a more general notion of a geodesic in differential geometry, which is determined by the local curvature: in our setting, a geodesic is the trace of an isometric embedding of a closed line segment in a general geodesic space, not necessarily in a manifold.

\begin{Def}\label{Def:rM} \cite{Haus}
Suppose $X$ is a geodesic space. Define $r(X) \geq 0$ as the least upper bound of the set of real numbers $r$ satisfying the following two conditions:
\begin{enumerate}
 \item For all $x,y\in X$ with $d(x,y)<2r$ there exists a unique  geodesic from $x$ to $y$ of length $2r$.
 \item Let $x,y,z,u\in X$ with $d(x,y)<r, d(u,x)<r, d(y,u)<r$ and let $z$ be a point on the geodesic joining $x$ and $y$. Then $d(u,z)\leq \max \{d(u,x), d(u,y)\}$.
 \item If $\gamma$ and $\gamma'$ are arc-length parameterized geodesics such that $\gamma(0)=\gamma'(0)$ and if $0 \leq s, s' < r$ and $0 \leq t < 1$, then $d(\gamma(ts),\gamma'(ts'))\leq d(\gamma(s), \gamma'(s'))$.
\end{enumerate}
\end{Def}

As was stated in \cite{Haus}, $r(X)>0$ if $X$ admits a strictly positive injectivity radius and an upper bound on its sectional curvature. In particular, each compact Riemannian manifold has $r(X)>0$. 

%
\begin{Def}
 \label{Def:GeodConvex}
 Suppose $X$ is a geodesic space with $r(X)>0$, $x_0\in X$, and $A\subset B(x_0, r(X))$ is open. Subset $A$ is geodesically convex if each  geodesic  connecting two points in $A$ lies entirely in $A$.  The geodesic hull of $A'\subset B(x_0, r(X))$  is the smallest geodesically convex set containing $A'$.
 \end{Def}

\begin{Prop}
 \label{Prop:Convex}
 Suppose $X$ is a geodesic space with $r(X)>0$, $\eps>0$, $0<q\leq r(X)/2$, and $A\subset X$ is of diameter $\diam(A)\leq q-2\eps$. Then there exists an open geodesically convex subset $A_\infty$ containing $A$ of diameter  $\diam(A_\infty)\leq q-\eps$.
\end{Prop}

\begin{proof}
We will construct $A_\infty$ inductively. Define $A_0=A$ and let $A'_0$ be the union of traces of all unique geodesics from  Definition \ref{Def:rM}(1), whose endpoints lie in $A_0$. It follows from    Definition \ref{Def:rM}(2)  that $\diam(A_0)=\diam(A'_0)$.

We now present an inductive step. For each $i\in \NN$:
\begin{itemize}
 \item Define $A_i$ to be the open $2^{-i-1}\eps$ neighborhood of $A'_{i-1}$. Note that $\diam(A_i) \leq \diam(A'_{i-1})+ 2^{-i} \eps \leq \diam(A_0) +(1-2^i)\eps \leq q-\eps -2^{-i}\eps$.
 \item Define $A'_i$ as the union of traces of all unique 
 geodesics from  Definition \ref{Def:rM}(1), whose endpoints lie in $A_i$. Since $\diam(A_i)$ is less than $r(X)/2$ we can use Definition \ref{Def:rM}(2)  to conclude $\diam(A_i)=\diam(A'_i)$.
 \end{itemize}
Then $A_\infty =\cup_{i\in \NN} A_i=\cup_{i\in \NN} A'_i$ satisfies the conclusions of the proposition.
\end{proof}

\begin{Thm} \label{Thm:Hausmann}
[Hausmann's Theorem \cite{Haus}]
Suppose $X$ is a geodesic space with $r(X)>0$ (for example, a compact Riemannian manifold). Then $X\simeq \Rips(X,q)$, for each positive $q\leq r(X)/2$.
\end{Thm}

\begin{proof}
 For each finite subset $\sigma \subset X$ of diameter less than $q$ choose a geodesically convex open subset $\sigma_\infty$ of diameter less than $q$ by Proposition \ref{Prop:Convex}. Define an open cover consisting of such sets:
$$
\U=\{\sigma_\infty\}_{\sigma \subset X, |\sigma|<\infty, \diam(\sigma)<q}.
$$ 
By Theorem \ref{Thm:RipsNerve}, $\N(\U)\simeq \Rips(X,q)$. Cover $\U$ is a good cover since non-empty intersections of geodesically convex sets are geodesically convex sets, which in our case are contractible by Definition \ref{Def:GeodConvex} as they are all of diameter less than $r(X)/2$. Hence $\N(\U)\simeq X$ by Theorem \ref{Thm:Nerve}.
\end{proof}

Another alternative proof of this fact was obtained by Henry Adams and Florian Frick.

In a similar way we can prove a version the Hausmann's Theorem for closed Rips complexes.

\begin{Thm} \label{Thm:Hausmann1}
[A version of the Hausmann's Theorem for closed Rips complexes]
Suppose $X$ is a geodesic space with $r(X)>0$ (for example, a compact Riemannian manifold). Then $X\simeq \cRips(X,q)$, for each positive $q< r(X)/2$.
\end{Thm}

\begin{proof}
 The proof goes along the same lines as the proof of Theorem \ref{Thm:Hausmann}, with $\U$ of that proof being modified by adding geodesic hulls of the collection $\{\sigma \subset X :  |\sigma|<\infty, \diam(\sigma)=q\}$. After the modification the obtained cover is still good (as we only added geodesically convex sets of diameter less than $r(x)/2$),  numerable (as it contains a numerable cover) and the diameters of the added sets are  precisely $q$ by Definition \ref{Def:rM} (2).
\end{proof}

\subsection{Reconstruction of metric graphs}
A \textbf{metric graph} $X$ is a finite graph equipped with a geodesic metric. The length of the shortest non-contractible loop in a metric graph will be denoted by $\ell(X)$. If $X$ is a tree then $\ell(X)=\infty$. 
In this subsection we use the Reconstruction Theorem for Rips complexes to completely classify scales $r$, for which the Rips complex of a metric graph is homotopy equivalent to $X$. 

\begin{Prop}
 \label{Prop:MetricGraph}
 Suppose $X$ is a metric graph. 
\begin{enumerate}
 \item If  $d(x,y) < \ell(X)/2$ for some $x,y\in X$, then there exist precisely one simple (non-selfintersecting) path $\gamma$ from $x$ to $y$ of length less than $\ell(X)/2$. In particular, a shortest geodesic between $x$ and $y$ is unique and equals $\gamma$.
 \item If $A\subset X$ is geodesically convex and $\diam(A)< \ell(X)/2$, then $A$ is a tree.
 \item Let $A\subset X$ be compact,  $\diam(A)< \ell(X)/3$, $x,y\in A$, and let $\gamma$ be the (trace of the) unique shortest geodesic between $x$ and $y$. Then  $\diam (A \cup \gamma) = \diam(A)$. 
 \item Let $A\subset X$ be finite,  $\diam(A)< \ell(X)/3$, and for $x,y\in A$ let $\gamma_{x,y}$ be the (trace of the) unique shortest geodesic between $x$ and $y$. Define $A'=\cup_{x,y\in A}\gamma_{x,y}$. Then $A'$ is compact, geodesically convex and $\diam (A') = \diam(A)$. In particular, by (2) $A'$ is a tree.  
\end{enumerate}
\end{Prop}

\begin{proof}
(1)
If this was not so, the two different paths between them would form a non-contractible loop of length less than $ \ell(X)$, a contradiction. 

(2) It suffices to show that $A$ contains no non-contractible loop. Assume $\gamma$ is the shortest non-contractible loop in $A$ and suppose $x,y\in \gamma \subset A$ are diametrally opposite points on $\gamma$, i.e., both segments $\gamma_1$ and $\gamma_2$ of $\gamma$ between $x$ and $y$ are of the same length at least $\ell(X)/2$. The shortest geodesic $\nu$ between $x$ and $y$ is of length less than $\ell(X)/2$ and can't be homotopic to both $\gamma_1$ and $\gamma_2$ (rel endpoints). Without the loss of generality we assume $\nu$ is not homotopic to $\gamma_1$. Concatenating $\nu$ and $\gamma_1$ we obtain a non-contractible loop in $A$ of length at most $\diam(A) + \textrm{length}(\gamma) /2 < \textrm{length}(\gamma)$, a contradiction to the minimality of $\gamma$.

(3) Choose $z\in A$. Let $\gamma_x$ and $\gamma_y$ denote the geodesics from $z$ to $x$ and $y$ respectively. Both $\gamma_x$ and $\gamma_y$ follow the same trace to some point on $\gamma$, and then each of them runs along $\gamma$ towards its respective point: if this was not the case, $\gamma, \gamma_x$ and $\gamma_y$ would form a non-contractible loop of length less than $\ell(X)$, a contradiction. Geodesics $\gamma_x$ and $\gamma_y$ jointly visit all the points of $\gamma$, hence for each $w\in \gamma$ we have $d(z,w)\leq \diam(A)$.

(4) We proceed by induction. Let $A=\{x_1, x_2, \ldots, x_k\}$ and $A_i=\{x_1, x_2, \ldots, x_i\}$. Now $A'_2$ is geodesically convex of diameter $d(x_1, x_2)=\diam(A_2)$ as it is just a geodesic. 

Assume now that $A'_j$ is geodesically convex of diameter $\diam(A_j)$. By inductive use of (3) we conclude $\diam(A'_{j+1})=\diam(A_{j+1})$. Furthermore, the argument of (3) shows that the traces of geodesics $\{\gamma_{j+1, i}\}_{i \leq j}$ cover $A'_j$ hence by (1), $A'_{j+1}$ is geodesically convex.
\end{proof}

\begin{Thm}
\label{Thm:RipsGraph}
 Suppose $X$ is a metric graph and $r>0$. Then $\Rips(X,r) \simeq X$ iff $r \leq \ell(X)/3$.
\end{Thm}

\begin{proof}
Using the notation of Proposition \ref{Prop:MetricGraph}(4) define a closed cover of $X$ in the following way: 
$$
\U =\{U' \subset X \mid U\subset X, U \textrm{ finite}, \diam(U)<r\}.
$$ 
By Theorem \ref{Thm:RipsNerve1} it suffices to show that $\U$ is a good numerable cover of $X$. Cover $\U$ is numerable as its interiors cover $X$. Furthermore, each element of $\U$ is geodesically convex of diameter less than $\ell(X)/3$, hence so are all finite intersections of elements of $\U$. Proposition \ref{Prop:MetricGraph}(2) then implies $\U$ is a good cover.

For $r > \ell(X)/3$ the first homology groups of $\Rips(X,t)$ and $X$ are known to differ, see the main results of \cite{ZV} or \cite{7A} for details.
\end{proof}

Again, the same argument can be constructed for closed Rips complexes to prove the following result.

\begin{Thm}
\label{Thm:RipsGraphClosed}
 Suppose $X$ is a metric graph and $r>0$. Then $\cRips(X,r) \simeq X$ iff $r < \ell(X)/3$.
\end{Thm}

\begin{Rem}
Suppose $X$ is a metric graph and $r>0$. Analogous arguments to those in the proof of Theorem \ref{Thm:RipsGraph} could be made to prove that  $\Cech(X,r) \simeq X$ iff $r \leq \ell(X)/4$. Similarly,   $\cCech(X,r) \simeq X$ iff $r < \ell(X)/4$.
\end{Rem}

\section{Functorial Dowker-Nerve diagram}
\label{Sect:Funct}

The results of the previous sections relate to complexes obtained at a fixed scale. In the context of filtrations however, we are interested in preserving or relating these results through various scales. With this aim we present in this section a functorial version of the Nerve Lemma  coupled with functorial Dowker duality spanning various scales. We call the result the Functorial Dowker-Nerve diagram, as it combines functorial versions of the Dowker Theorem and of the Nerve lemma.

Functorial versions of these results have been considered before. A functorial version of the Nerve Lemma appears in \cite{CO}  and later in \cite{CM} for pairs of finite good open covers of paracompact spaces. Approximate homological versions were obtained in \cite{GS} and \cite{CS}. On the other hand, a functorial version of the Dowker duality was proved in \cite{CM}. The functorial versions of the Nerve Lemma (Lemma \ref{Lem:FunctorialNerve}) and of the Dowker duality (Theorem \ref{Thm:FDowker1}) presented in this paper are more general than the previously known versions.

We first introduce a notation of our functorial setting for the rest of this section. Let $f\colon X \to Y$ be a map. Suppose $\U=\{U_\alpha\}_{\alpha \in A}$ and $\W=\{W_\beta\}_{\beta \in B}$ are numerable covers of $X$ and $Y$ respectively, so that $f(\U)$ is a refinement of $\W$, i.e., 
  $\forall \alpha\in A, \ \exists \beta_\alpha \in B: f(U_\alpha) \subset W_{\beta_\alpha}$.

Define the following maps: 
\begin{itemize}
 \item $\f_\N \colon \N(\U) \to \N(\W)$ is the simplicial map  mapping $U_\alpha \mapsto W_{\beta_\alpha}$. While $\f_\N$ depends on the choice of pairing $\alpha \mapsto \beta_\alpha$ (each  $f(U_\alpha)$ can typically be contained in more than one element of $\W$), its homotopy type does not. To see this note that a different choice of pairing induces a contiguous (hence homotopic) map.
 \item $\f_\V \colon \V(\U) \hookrightarrow \V(\W)$ is the simplicial map mapping $x \mapsto f(x)$.
 \item maps $i_\U \colon X \to \N(\U)$ and $i_\W \colon Y \to \N(\W)$ arise via locally-finite partitions of unity subordinated to $\U$ and $\W$ as explained in Section \ref{Sec:PartUnity}.
\end{itemize}

\begin{Lem}
\label{Lem:FunctorialNerve}
 [Functorial Nerve Lemma]
 The following diagram commutes up to homotopy:
  $$
 \xymatrix{
 Y \ar[r]^{i_\W}&  \N(\W)  \\
 X \ar[r]_{i_\U} \ar[u]^{f}& \N(\U) \ar[u]_{\f_\N}
 }
 $$
 Furthermore, if $\U$ and $\W$ are good covers, then $i_\W$ and $i_\U$ are homotopy equivalences, hence $f$ is a homotopy equivalence iff $\f_\N$ is.
\end{Lem}

\begin{proof}
 Maps $i_\W \circ f$ and $\f_\N \circ i_\U$ are contiguous hence homotopic. For the second part use Theorem \ref{Thm:Nerve}. 
\end{proof}

We will also require a version of the Dowker Theorem, which is functorial in our setting described above. 

\begin{Thm}\label{Thm:FDowker1}
  [The Functorial Dowker Theorem  for covers].
 There exist homotopy equivalences $\tilde \gamma_\U$ and $\tilde \gamma_\W$ for which the following diagram commutes up to homotopy:
 $$
 \xymatrix{
 \N(\W) &\V(\W) \ar[l]_{\tilde \gamma_\W}\\
 \N(\U)  \ar[u]_{\f_\N}&\V(\U) \ar[u]_{\f_\V}\ar[l]^{\tilde\gamma_\U}}
 $$
 The statement is true even if $\U$ and $\W$ are not numerable.
\end{Thm}

\begin{proof}
For any $C \subset X$ let $\Delta_C $ denote the full simplicial complex on $C$, i.e., each finite subset of $C$ spans a simplex in $\Delta_C$. Recall that full simplicial complexes are contractible and note that if $C\subset U$ for some $U\in \U$, then $\Delta_C \subset \V(\U)$.
 
We see that  $\{\Delta_{U_\alpha}\}_{\alpha\in A}$ is a good cover of $\V(\U)$ as for each $A'\subset A$ the intersection
 $$
 \bigcap_{\alpha \in A'}\Delta_{U_\alpha}=\Delta_{(\cap_{\alpha \in A'}U_\alpha)}
 $$
 is a full simplicial complex. However, it may not be numerable, as its interiors may not cover $\V(\U)$. To remedy this shortcoming we modify it by slightly thickening the elements of the cover. For a moment let us assume that each simplex of $\V(\U)$ is isometric to the standard simplex. We enlarge each $\Delta_{U_\alpha}$ to $\tilde\Delta_{U_\alpha}$ so that for each simplex $\sigma\in \V(\U)$, 
 $$
 \tilde\Delta_{U_\alpha} \cap \sigma = N(\Delta_{U_\alpha}, 0.1) \cap \sigma,
 $$
 i.e, we thicken the sets by $0.1$ in each adjacent simplex. Note that
for each $A'\subset A$ the intersection $\cap_{\alpha \in A'}\tilde \Delta_{U_\alpha}$ deformation contracts to $\cap_{\alpha \in A'}\Delta_{U_\alpha}$. In particular, $\cap_{\alpha \in A'}\tilde \Delta_{U_\alpha}= \emptyset $ iff $\cap_{\alpha \in A'}\Delta_{U_\alpha}= \emptyset $ and $\widetilde \U = \{\tilde \Delta_{U_\alpha}\}_{\alpha \in A}$ is a good numerable cover of $\V(\U)$. Furthermore, $\N(\U)$ is isomorphic to $\N(\widetilde \U)$ by identifying vertices $U \in \N(\U)$ with $\tilde \Delta_{U}\in \N(\widetilde \U)$. Using this identification and Lemma \ref{Lem:FunctorialNerve}  we define a homotopy equivalence $\tilde \gamma_\U = i_{\widetilde U}\colon \V(\U) \to \N(\U)$. In the same way we define a  homotopy equivalence $\tilde \gamma_\W = i_{\widetilde W}\colon \V(\W) \to \N(\W)$.
As $\f_\N (\widetilde U)$ is a refinement of $\widetilde W$, 
 the diagram commutes up to homotopy by Lemma \ref{Lem:FunctorialNerve}.

Throughout the proof we never required $\U$ or $\W$ to be numerable.
\end{proof}

\begin{Rem}
 Note that Theorem \ref{Thm:FDowker1} is stronger than the Functorial Dowker Theorem of \cite{CM}. The setting of \cite{CM} requires nested relations on a fixed product $Y \times Z$ which, by the theory of Section \ref{Sect:Dowker}, translates to covers $\U=\{U_\alpha\}_{\alpha \in A}$ and $\W=\{W_\alpha\}_{\alpha \in A}$ of $X$ with $U_\alpha \subset W_\alpha$. In particular, in that setting the covers have to be indexed by the same set, nesting has to precisely respect the indexing, and covers have to be of the same space (hence  function $f$ is the identity). 
\end{Rem}

We may now combine the results of Lemma \ref{Lem:FunctorialNerve} and Theorem \ref{Thm:FDowker1} into a single diagram. For the sake of completeness we also include a description of the setting of this section. 

\begin{Thm}
 [Functorial Dowker-Nerve diagram]
 \label{Thm:NerveNested}
Let $f\colon X \to Y$ be a map. Suppose $\U=\{U_\alpha\}_{\alpha \in A}$ and $\W=\{W_\beta\}_{\beta \in B}$ are numerable covers of $X$ and $Y$ respectively, so that
  $\forall \alpha\in A, \ \exists \beta_\alpha \in B: f(U_\alpha) \subset W_\beta$.

Define the following maps: 
\begin{itemize}
 \item $\f_\N \colon \N(\U) \to \N(\W)$ is the simplicial map  mapping $U_\alpha \mapsto W_{\beta_\alpha}$;
 \item $\f_\V \colon \V(\U) \hookrightarrow \V(\W)$ is the simplicial map  mapping $x \mapsto f(x)$;
  \item maps $i_\U \colon X \to \N(\U)$ and $i_\W \colon Y \to \N(\W)$ arise via locally-finite partitions of unity subordinated to $\U$ and $\W$;
 \item homotopy equivalences $\gamma_\U \colon \N(\U)\to \V(\U)$ and $\gamma_\W \colon \N(\W)\to \V(\W)$ arise from Theorem \ref{Thm:FDowker1} as homotopy inverses of $\tilde \gamma_\U$ and $\tilde \gamma_\W$.
\end{itemize}
Then the following diagram commutes up to homotopy:
 $$
 \xymatrix{
 Y \ar[r]^{i_\W}&  \N(\W) \ar[r]^{\gamma_\W}&\V(\W) \\
 X \ar[r]_{i_\U} \ar[u]^{f}& \N(\U) \ar[r]_{\gamma_\U} \ar[u]_{\f_\N}&\V(\U) \ar[u]_{\f_\V}}
 $$
Furthermore, if $\U$ and $\W$ are good covers, then $i_\U$ and $i_\W$ are homotopy equivalences. In this case the following holds: if any  of the maps $\{f,\f_\N, \f_\V\}$ is a homotopy equivalence, then all three are.
\end{Thm}

\begin{Rem}
A consequence of Theorem \ref{Thm:NerveNested} is that in the case of good numerable covers $\U$ and $\W$, $f$ is a homotopy equivalence iff $\f_\N$ or $\f_\V$ is. In effect this means that we can decide whether $f$ is a homotopy equivalence by considering (testing) only the induced maps on the nerve or Vietoris complex.
\end{Rem}

By  Theorem \ref{Thm:NerveNested}, the existence of good covers at small scale implies all sorts of homotopy equivalences between the induced complexes. Here we provide some of them in two specific cases. 

\begin{Thm}
 [Local Rips-\v Cech coincidence]
 \label{Thm:RipsCech}
 Suppose $X$ is: 
\begin{enumerate}
 \item a metric graph and $0<r\leq\ell(X)/4$, or
\item   a geodesic space with $r(X)>0$ and $0\leq r(X)/2$.
\end{enumerate}
Then the inclusions
$\Cech(X,r/2) \hookrightarrow \Rips (X,r) \hookrightarrow \Cech(X,r)$ are homotopy equivalences. 

The same statement for closed complexes holds if $0<r<\ell(X)/3$, or $0< r(X)/2$ respectively.
\end{Thm}

\begin{proof}
Let $\U^r$ and $\U^{r/2}$ be the covers of $X$ by all open balls of radius $r$ and $r/2$ respectively. Note that $\Cech(X,r/2)=\N(\U^{r/2})=\V(\U^{r/2})$ and $\Cech(X,r)=\N(\U^{r})=\V(\U^{r})$. Let $\W$ be a good open cover of $X$ from Theorems \ref{Thm:Hausmann} or \ref{Thm:RipsGraph} respectively, so that $\Rips(X,r)= \V(\W)$.
The conclusion holds by Theorem \ref{Thm:NerveNested} since maps $\f_\V$ are inclusions and the covers are appropriately nested.
 \end{proof}

\begin{Rem}
 The passages of parameter from $r/2$ to $r$ to $r$ in $\Cech(X,r/2) \hookrightarrow \Rips (X,r) \hookrightarrow \Cech(X,r)$ are essentially due to two facts: 
\begin{itemize}
 \item  that balls of radius $r/2$ are of diameter $r$;
 \item that subsets of diameter $r$ are contained in a ball of radius $r$. 
\end{itemize}
The second parameter passage can be tightened in specific cases (for example in Euclidean spaces) by the versions of the Jung Theorem.
\end{Rem}

\begin{Thm}
 [Initial persistence invariance for Rips complexes]
 Suppose $X$ is: 
\begin{enumerate}
 \item a metric graph and $0<r_1 \leq r_2 < r_3=\ell(X)/3$, or
\item   a geodesic space with $r(X)>0$ and $0<r_1 \leq r_2 < r_3=r(X)/2$.
\end{enumerate}
Then the all maps in the following  (up to homotopy commutative) diagram are homotopy equivalences between spaces
$$
\xymatrix{
&&
\Rips(X,r_3)& \\
&&
\Rips(X,r_2) \ar@{^(->}[r] \ar@{^(->}[u]
&
\cRips(X,r_2)\ar@{_(->}[ul]
\\
X \ar[rr]_{\gamma_\U \circ i_\U}
 \ar[urr]_{\gamma_\W \circ i_\W}
  \ar[uurr]^{\gamma_\W' \circ i_\W'}&&
 \Rips(X,r_1) \ar@{^(->}[r] \ar@{^(->}[u]
 &
\cRips(X,r_1) \ar@{_(->}[ul] \ar@{^(->}[u]
}
$$
with maps $\gamma_\U \circ i_\U$, $\gamma_\W \circ i_\W$, and $\gamma_\W' \circ i_\W'$ arising from Theorem \ref{Thm:NerveNested} for appropriate covers described in Section \ref{Sect:Nerve}.
\end{Thm}

\begin{proof}
Follows from Theorem \ref{Thm:NerveNested} for appropriate covers described in Section \ref{Sect:Nerve}.
\end{proof}

\begin{Rem}
 The same theorem holds for \v Cech complexes as well with $X$ being a:
 \begin{enumerate}
 \item a metric graph and $0<r_1 \leq r_2 < r_3=\ell(X)/4$, or
\item   a geodesic space with $r(X)>0$ and $0<r_1 \leq r_2 < r_3=r(X)/4$.
\end{enumerate}
\end{Rem}

\begin{Rem}
While the initial persistence invariance results don't seem to differ much in the case of open filtrations as compared to the closed filtrations, the understanding of the underlying theoretical difference could have important consequence on computational setting, i.e., the case where such filtrations are approximated by finite subcomplexes. The results of \cite{AA} suggest that the open and closed filtrations of a geodesic circle differ by ephemeral summands, i.e., they differ only at a discrete set of points. These summands however seem to grow in size when approximated by finite sample \cite{AA}. A similar effect was observed in  \cite{ZV2}. We plan to delve deeper in that direction in the future research. 
\end{Rem}

\section{Reconstruction results by subsets}
\label{Sect:subsets}

Many of the results of the previous sections were aimed at reconstruction of the homotopy type of a space in terms of its Rips complex. In this section we focus on the reconstruction of $X$ or its homology using Rips complexes of subspaces.
 
Theorem \ref{Thm:C} states that the homotopy type of a sufficiently nice compact space can be reconstructed as the Rips complex of a countable subset $C$. It is an open question whether in the context of Theorem \ref{Thm:C}, the subset $C$ can be taken to be a finite subset of $X$. Such a result is known to hold for sufficiently nice Riemannian manifolds \cite{Lat} and certain subsets of the Euclidean space \cite{Att}. However, the currently known proofs of both of these results rely strongly on the underlying structure of the space (either a manifold structure or that of certain Euclidean neighborhoods). 

Let us introduce some new notation: if $\U$ is a cover of $X$ and $A\subset X$, then $\U_A=\{U \cap A \mid U \in \U\}$ is a cover of $A$.

\begin{Thm}
 \label{Thm:C}
Suppose $X$ is a compact metric space, $r>0$ and $\U$ is a good open cover of $X$ satisfying the following condition for	each finite subset $\sigma \subset X$:  $\sigma$ is contained in some element of $\U$ iff $\diam(\sigma)<r$. Then there exists a countable $C \subset X$ satisfying $\Rips(C,r)\simeq X$.
\end{Thm}

\begin{proof}
Note that  $\Rips(X,r)\simeq X$ by Theorem \ref{Thm:RipsNerve1}. 

Let $\U_1 \subset U$ be a finite subcover of $\U$ and choose a finite subset $X_1\subset X$ so that each non-trivial intersection of elements of $\U_1$ contains a point of $X_1$, i.e., $\N(\U_1|_{X_1})=\N(\U_1)$. 

We proceed by inductive definition of covers $\U_n$ and finite subsets $X_n$ for $n>1$:
\begin{itemize}
 \item Let $\U_n$ be a finite subcover of $\U$ containing $\U_{n-1}$, so that for each finite $\sigma \subset X_{n-1}$ of diameter less than $r$ there exists $U_\sigma \in \U_n$ containing $\sigma$. Consequently $\V(\U_n|_{X_{n-1}})=\Rips(X_{n-1},r)$.
 \item Let $X_n\subset X$ be a finite subset containing $X_{n-1}$, so that each non-trivial intersection of elements of $\U_{n}$ contains a point of $X_n$, i.e., $\N(\U_n|_{X_n})=\N(\U_n)$. 
\end{itemize}

Define $C=\cup_n X_n$ and note that $\Rips(C,r)=\cup_n \Rips(X_n,r)$. We can now construct the following diagram:

$$
\xymatrix
{
X \ar@{=}[dd] \ar[rr]&& 
\N(\U_1)=\N(\U_1|_{X_1}) \ar@{^(->}[d]\ar[rr]&& 
\V(\U_1|_{X_1})\ar@{^(->}[d]
\\
 && 
 \N(\U_2|_{X_1}) \ar@{^(->}[d] \ar[rr]&& 
 \V(\U_2|_{X_1}) = \Rips(X_1, r)\ar@{^(->}[d]
 \\
X \ar[rr]&& 
\N(\U_2)=\N(\U_2|_{X_2}) \ar@{^(->}[d]\ar[rr]
&& \V(\U_2|_{X_2})\ar@{^(->}[d]
\\
 && 
 \N(\U_3|_{X_2}) \ar[rr]&& 
 \V(\U_3|_{X_2})=\Rips(X_2,r)\ar@{.>}[d]
 \\
 && && \Rips(C,r)
}
$$
By Theorem \ref{Thm:NerveNested} the diagram commutes up to homotopy,  all horizontal maps are homotopy equivalences, and all inclusions $ \V(\U_{n+1}|_{X_{n}}) \to \V(\U_{n+2}|_{X_{n+1}})$ are homotopy equivalences. Consequently $ \cup_n \V(\U_{n+1}|_{X_{n}}) \simeq X$, which proves the theorem as $\cup_n \V(\U_{n+1}|_{X_{n}}) =\cup_n \Rips(X_n,r)=\Rips(C,r)$.
\end{proof}

\begin{Rem}
 Theorem \ref{Thm:C} actually holds even if $X$ is separable instead of compact, with only a minor adjustment to the proof. 
\end{Rem}

While it is unknown whether set $C$ of Theorem \ref{Thm:C} can be taken to be finite (i.e., whether we can construct the homotopy type of $X$ as the Rips complex of a finite subset), we can make use of a trick from  \cite{CO}  to reconstruct the homology of $X$ through nested finite Rips complexes in Theorem \ref{Thm:H}.

\begin{Lem}\label{Lem:CO}
 [Excerpt from Lemma 3.2 of \cite{CO}] If $A \to B \to C \to E \to F$ is a sequence of homomorphisms such that $\rank (A \to F) = \dim C$, then $\rank (B \to E) = \dim C$.
 
\end{Lem}

\begin{Thm}
 \label{Thm:H}
 Suppose $X$ is a compact metric space, $r>0$ and $\U$ is a good open cover of $X$ satisfying the following condition for	each finite subset $\sigma \subset X$:  $\sigma$ is contained in some element of $\U$ iff $\diam(\sigma)<r$. Then there exist finite subsets $X_1 \subset X_2 \subset X$ so that for each $k\in \NN$ and each field $\mathbb F$, $H_k(X,\mathbb F) \cong \im j_*$, where $j_* \colon H_k(\Rips(X_1,r),\mathbb F) \to H_k(\Rips(X_2,r),\mathbb F)$ is the inclusion induced homomorphism.
 \end{Thm}
 
\begin{proof}
Apply Lemma \ref{Lem:CO} to the right vertical thread of the diagram from the proof of Theorem \ref{Thm:C}.
\end{proof}

\begin{Rem}
Note that the nesting of Theorem \ref{Thm:H} is horizontal, that is, it only considers one scale $r$ and increases the subspace (sample). This is in contrast to the use in  \cite{CO} where nesting is vertical (i.e., includes multiple scales at fixed sample), confined to small $r$, and constrained by the weak feature size of \cite{CO} (and consequently the structure of the \v Cech complex in Euclidean space). Therefore Theorem \ref{Thm:H} can be used at any individual scale without  extra requirements on the ambient space (provided the conditions of Theorem \ref{Thm:H} hold) and irrespectively of other scales. In particular, we may get appropriate reconstruction even for scales above the weak feature size. For a demonstration see Example \ref{Ex:1}.
\end{Rem}

\begin{Example}
\label{Ex:1}
\begin{figure}
\begin{tikzpicture}[scale=.9]
\draw (-3.3,0)++(150:.3) arc (150:-150:.3);
\draw (0,0)--(-3,0);
\draw (0,0) arc (180:540:1);
\end{tikzpicture}
\caption{Space $X$ of Example \ref{Ex:1}.}
\label{Fig:1}
\end{figure}
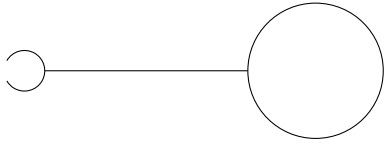

Let $\eps < \frac{1}{10}$. Space $X\subset \RR^2$ is presented by Figure \ref{Fig:1}: a smaller circle of diameter $1$ has an $\eps$-gap on the left and is connected on the right by a line of length $10$ to a circle of diameter $D>3$. It is easy to verify that by our results $\Rips(X,r)\simeq X$ for any $r < \eps$. A similar  reconstruction can be obtained by \cite{Att}: as $\eps/2>0$ is the weak feature size of $X$, we can choose a finite $A\subset X$  and a small $r>0$ so that   $\Rips(A,r)\simeq X$. Furthermore, the homology of $X$ with coefficients in a field can be extracted by \cite{CO} using a pair of Rips complexes of a finite subset at different scales or by using Theorem \ref{Thm:H}. However, it is not hard to see that also for $R\in (1,D/3)$ there exists a good $\Rips_R$-cover of $X$, hence $\Rips(X,R)\simeq X$ for scale $R$ as well. Consequently the homotopy type of $X$ can be extracted using Theorem \ref{Thm:C} at scale $R$ by a countable subset. 
Similarly, its homology with coefficients in a field may be extracted using Theorem \ref{Thm:H} at scale $R$ using two nested finite subsets.  Since the interval $(1,D/3)$ between two critical values of the distance function $x \mapsto d(x,X)$ can be arbitrarily small, results of \cite{CO, Att} do not apply in this case for sufficiently small $D>3$.
\end{Example}

\section{Appendix: homotopy types of Rips complexes }
\label{sect:App}

Suppose $X\subset \RR^n$ is a finite subset and $r>0$. A convenient implication of the nerve theorem for \v Cech complexes is that $\Cech_{\RR^n}(X,r)$ is homotopy equivalent to a subspace of the ambient space $\RR^n$ (in particular to the $r$-neighborhood of $X$). It is easy to see, however, that the same does not always hold for Rips complexes. In this section we explain why this is the case despite the presented version of the nerve theorem for Rips complexes, and how the difficulties surrounding this issues can sometimes be circumvented. We will do that by analyzing the case of a regular planar hexagon using Theorem \ref{Thm:RipsNerve}. A general strategy is the following: for $r>0$ first use Theorem \ref{Thm:RipsNerve} to construct a cover $\U$ of $X$ with $\Rips(X,r)= \V(\U)$, and then replace each element of $\U$ by a convex set in $\RR^n$ to obtain a good cover $\widetilde \U$, for which $\Rips(X,r) \simeq \cup_{\widetilde U \in \widetilde \U} \widetilde U$ potentially holds.

Let $X_1=\{x_1, x_2, \ldots, x_6\}\subset \RR^2$ denote the vertices of a regular hexagon in the plane with the side length $1$, i.e., $d(x_i,x_{i+1})=1=d(x_1, x_6)$. Define $r_2=d(x_1,x_3)$ and $r_3=d(x_1,x_4)$. The open Rips complexes attain four different homotopy types.
\begin{description}
 \item[a) $r \leq 1$] $\Rips(X_1,r) = \V(\{\{x_1\}, \{x_2\}, \ldots, \{x_6\}\})= X_1$ is a discrete space.
  \item[b) $1 <r \leq r_2$] $\Rips(X_1,r) = \V(\U_1)$ for 
  $$
  \U_1=\{\{x_1, x_2\}, \{x_2, x_3\}, \ldots, \{x_5,x_6\}, \{x_6,x_1\}\}.
  $$
  Replacing each element of $\U_1$ by the open $(r-1)/4$-neighborhood of its convex hull we get a good planar cover $\widetilde \U_1$ of the boundary of the corresponding regular hexagon in the plane. Since nonempty every intersection of sets of $\widetilde \U_1$ contains a point of $X_1$ we have 
  $\Rips(X_1,r) \simeq \N(\U_1) = \N(\widetilde \U_1) \simeq S^1$. 
  
In this case $\Rips(X_1,r)$ is actually homeomorphic to $S^1$. In order to demonstrate a more nontrivial example, let $X_2$ be obtained by replacing each $x_i$ by a subset $A_i$ of $n$ points in $B(x_i, (r-1)/4)$. In this case $\Rips(X_2,r)$ is  a $(2n-1)$-dimensional complex, but the above argument (with the open cover consisting of six sets, each of which is the open $(r-1)/4$-neighborhood of the convex hull of consecutive sets $A_i$) still holds hence $\Rips(X_2,r)\simeq S^1$.
  \item[c) $r_2 <r \leq r_3$] $\Rips(X_1,r) \simeq \N(\U_2)$ for 
  $$
  \U_2=\{\{x_1, x_2, x_3\}, \{x_2, x_3, x_4\}, \ldots, \{x_5,x_6, x_1\}, \{x_6,x_1, x_2\}\} \cup 
  $$
  $$
  \cup\{\{x_1, x_3, x_5\}, \{x_2, x_4, x_6\}\}.
  $$  
   Replacing each element of $\U_1$ by the open $(r-r_2)/4$-neighborhood of its convex hull we get a good planar cover $\widetilde \U_2$. However, $\N(\U_2) \neq \N(\widetilde \U_2)$ since the sets of $\widetilde \U_2$ corresponding to element $\{x_1, x_3, x_5\}, \{x_2, x_4, x_6\}\in \U_1$ intersect in $\RR^2$ but not in $X_1$. In fact $\N(\widetilde \U_2)$ is contractible while $ \N( \U_2)\simeq S^2$. To see the later note that if the vertices of $X_1$ are appropriately arranged as the vertices of an octahedron in $\RR^3$, then small neighborhoods of the faces of the octahedron form a good cover of the boundary of the octahedron so that:
\begin{itemize}
 \item the combinatorial structure of the cover matches that of $\U_2$;
 \item each nonempty intersection contains a point of $X_2$.
\end{itemize}
Hence $\Rips(X_1,r) \simeq \N(\U_2) \simeq S^2$. As before, the same holds if we replace each vertex of $X_1$ by a  large collection of points in its small neighborhood.
   
    \item[d) $r_3 <r $] $\Rips(X_1,r)$ is the full simplex on $X_1$ and hence contractible. 
\end{description}

Conclusion: For a nice representation of the homotopy type of $\Rips(X,r)$  one needs to:
\begin{itemize}
 \item choose a small cover $\U$ satisfying Theorem \ref{Thm:RipsNerve} (note that the cover of $X_2$ in \textbf{b)} above consists of only six sets despite $|X_2|=6n$), and
 \item extend $\U$ to a good cover $\widetilde \U$, so that the nonempty intersections of $\widetilde \U$ contain an element of $X$. 
\end{itemize}


\end{document}